\documentclass[a4paper,12pt]{amsart}

\usepackage{amssymb, enumitem}

\usepackage[breaklinks=true,colorlinks=true,linkcolor=green,citecolor=red,urlcolor=blue]{hyperref}

\allowdisplaybreaks

\newcommand \on{\overline{\nabla}}

\newcommand \la{\lambda}
\newcommand \ve{\varepsilon}
\newcommand \br{\mathbb{R}}
\newcommand \bO{\mathbb{O}}

\newcommand \Span{\operatorname{Span}}
\newcommand \id{\operatorname{id}}
\newcommand \<{\langle}
\renewcommand \>{\rangle}
\newcommand \ip{\< \cdot, \cdot \>}

\newcommand \Tr{\operatorname{Tr}}

\newcommand \tM{\overline{M}}
\newcommand \tc{\overline{c}}
\newcommand \tR{\overline{R}}

\newcommand \gv{\mathfrak{v}}
\newcommand \gz{\mathfrak{z}}
\newcommand \ga{\mathfrak{a}}
\newcommand \gs{\mathfrak{s}}
\newcommand \cJ {J_{\gz}}

\newcommand \Sh{\mathrm{S}}

\theoremstyle{plain}
\newtheorem{theorem}{Theorem}
\newtheorem*{theorem*}{Theorem}
\newtheorem*{corollary*}{Corollary}

\newtheorem*{conj*}{Conjecture}
\newtheorem{lemma}{Lemma}
\newtheorem{proposition}{Proposition}
\newtheorem*{prop*}{Proposition}

\theoremstyle{definition}

\newtheorem*{definition*}{Definition}

\theoremstyle{remark}
\newtheorem*{remark*}{Remark}
\newtheorem{example}{Example}

\begin{document}

\title[Submanifolds of harmonic manifolds]{Totally geodesic submanifolds of Damek-Ricci spaces and Einstein hypersurfaces of the Cayley projective plane} 

\author{Sinhwi Kim}
\address{Department of Mathematics, Sungkyunkwan University, Suwon, 16419, Korea}
\email{kimsinhwi@skku.edu}

\author{Yuri Nikolayevsky}
\address{Department of Mathematics and Statistics, La Trobe University, Melbourne, Victoria, 3086, Australia}
\email{y.nikolayevsky@latrobe.edu.au}

\author{JeongHyeong Park}
\address{Department of Mathematics, Sungkyunkwan University, Suwon, 16419, Korea}
\email{parkj@skku.edu}

\thanks {The second and the third author were supported by Basic Science Research Program through the National Research Foundation of Korea(NRF) funded by the Ministry of Education (NRF-2016R1D1A1B03930449). The second author was partially supported by ARC Discovery grant DP130103485.}

\subjclass[2010]{Primary 53C25, 53C30, 53B25; Secondary 53C35}
\keywords{harmonic manifold, totally geodesic submanifold, Damek-Ricci space, Einstein hypersurface}



\begin{abstract}
We classify totally geodesic submanifolds of Damek-Ricci spaces and show that they are either homogeneous (such submanifolds are known to be ``smaller" Damek-Ricci spaces) or isometric to rank-one symmetric spaces of negative curvature. As a by-product, we obtain that a totally geodesic submanifold of any \emph{known} harmonic manifold is by itself harmonic. We prove that the Cayley hyperbolic plane admits no Einstein hypersurfaces and that the only Einstein hypersurfaces in the Cayley projective plane are geodesic spheres of a particular radius; this completes the classification of Einstein hypersurfaces in rank-one symmetric spaces. We also show that if a $2$-stein space admits a $2$-stein hypersurface, then both are of constant curvature, under some additional conditions.
\end{abstract}

\maketitle

\section{Introduction}
\label{s:intro}

In this paper, we study geometry of some distinguished submanifolds of harmonic manifolds. Recall that a Riemannian manifold is called \emph{harmonic} if a punctured neighbourhood of any point admits a harmonic function which depends only on the distance to the point. There are several equivalent definitions of harmonicity \cite{BTV}. It is easy to see that a flat space and rank-one symmetric spaces are harmonic. Moreover, a harmonic manifold is Einstein, and if the scalar curvature is non-negative, it is either flat or (locally) rank-one symmetric \cite{Sz}. If the scalar curvature is negative, there exist non-symmetric harmonic manifolds, the \emph{Damek-Ricci spaces} \cite{DR}. These spaces are solvable Lie groups with a special left-invariant metric which are one-dimensional extensions of the generalised Heisenberg groups \cite{BTV}. By the result of \cite{Heb}, any \emph{homogeneous} harmonic manifold is either flat, or rank-one symmetric, or is a Damek-Ricci space. Despite considerable effort, the question of whether there exist non-homogeneous harmonic manifolds remains open. For the current state of knowledge in the theory of harmonic manifolds we refer the reader to \cite{Kn} and references therein.

\medskip

The question we address in this paper can be informally stated as ``to what extent the property of a manifold to be harmonic is inherited by its submanifold?" More specifically, the first question we consider is whether a totally geodesic submanifold of a harmonic manifold is itself harmonic (it appears in particular in \cite[p.~467]{BPV}). It is well known that totally geodesic submanifolds of rank-one symmetric spaces are again rank-one symmetric, and so to answer our question for \emph{known} harmonic manifolds one has to consider totally geodesic submanifolds of Damek-Ricci spaces. The following theorem gives a classification of totally geodesic submanifolds of Damek-Ricci spaces (and as a by-product, implies that all of them are harmonic).

\begin{theorem} \label{t:tgDR}
  Let $M$ be a connected, totally geodesic submanifold of a Damek-Ricci space $S$ such that $\dim M \ge 2$. Then one of the following holds.
  \begin{enumerate}[label=\emph{(\arabic*)},ref=\arabic*]
    \item \label{it:smallerDR}
    The submanifold $M$ is \emph{homogeneous} totally geodesic. Then $M$ is given by Theorem~\ref{t:tgDRhomo} and is locally isometric to a ``smaller" Damek-Ricci space.

    \item \label{it:tgrk1}
    The submanifold $M$ is locally isometric to a rank-one symmetric space of negative curvature.
  \end{enumerate}
\end{theorem}

A totally geodesic submanifold of a Riemannian space is called \emph{homogeneous} if it is an orbit of a subgroup of the isometry group of the space. Homogeneous totally geodesic submanifolds of homogeneous spaces are much better studied and understood than ``generic" totally geodesic submanifolds. For a Damek-Ricci spaces, their classification is given in the following theorem (for unexplained terminology see Section~\ref{s:pre}).

\begin{theorem}[\cite{Rou}] \label{t:tgDRhomo}
A submanifold $M$ of a Damek-Ricci space $S$ passing through the identity is a homogeneous totally geodesic submanifold if and only if $M$ is \emph{(}locally\emph{)} a subgroup of $S$ and $T_eM = \ga \oplus \gv' \oplus \gz'$, where $\gv' \subset \gv, \, \gz' \subset \gz$ and $[\gv', \gv'] \subset \gz', \; J_{\gz'} \gv' \subset \gv'$.
\end{theorem}

In case \eqref{it:tgrk1} of Theorem~\ref{t:tgDR} we only give the isometry type of $M$, but not the description of how $M$ is positioned within $S$. The reason for that is the fact that in a general Damek-Ricci space, there can be many rank-one symmetric totally geodesic submanifolds. Some of them are homogeneous; those are \emph{well-positioned} in the sense of \cite{BTV}: the tangent space $T_eM$ is the direct sum of its intersections with $\ga, \gv$ and $\gz$. Furthermore, as a rank-one symmetric space has a much larger isometry group than a Damek-Ricci space and many more totally geodesic submanifolds, there are rank-one totally geodesic submanifolds of rank-one homogeneous totally geodesic submanifolds which are not homogeneous viewed as totally geodesic submanifolds of the ambient Damek-Ricci space and whose tangent spaces are not well-positioned. In Proposition~\ref{p:-1} in Section~\ref{s:tgDR} we give a complete description of totally geodesic submanifolds of constant curvature $-1$ in Damek-Ricci spaces.

\medskip

In the second part of the paper, we study hypersurfaces of harmonic manifolds. With an eye on the harmonicity condition, we consider Einstein hypersurfaces of harmonic manifolds. The classification of such hypersurfaces is a non-trivial task even for rank-one symmetric spaces. The first result dates back to 1938: by \cite[Theorem~7.1]{Fia}, an Einstein hypersurface in a space of constant curvature is locally either totally umbilical, or developable (of conullity $1$), or is the product of spheres of particular radii in the sphere. In the first two cases, the hypersurface has constant curvature. There are no Einstein (real) hypersurfaces in the complex projective space and in the complex hyperbolic space: for $\mathbb{C}P^m$ (where $m$ is the complex dimension), this is proved in \cite[Theorem~4.3]{Kon} assuming that $m \ge 3$ and the hypersurface is complete; for $\mathbb{C}H^m, \; m \ge 3$, in \cite[Corollary~8.2]{Mon}. These results remain true locally and for $m \ge 2$ \cite[Theorem~8.69]{CR}. By \cite[Corollary~1]{OP}, there are no Einstein (real) hypersurfaces in the quaternionic hyperbolic space. In contrast, the quaternionic projective space does admit an Einstein hypersurface: it is proved in \cite[Corollary~7.4]{MP} that a connected (real) hypersurface in $\mathbb{H}P^m$ (where $m \ge 2$ is the quaternionic dimension and the metric is normalised in such a way that the sectional curvature lies in $[1, 4]$) is Einstein if and only if it is an open, connected subset of a geodesic sphere of radius $r$, where $\cot^2 r = \frac{1}{2m}$.

To complete the classification of Einstein hypersurfaces in rank-one symmetric spaces it remains to consider the cases when the ambient space is either the Cayley projective plane $\bO P^2$ or its non-compact dual $\bO H^2$.

\begin{example} \label{ex:sphereO}
In the Cayley projective plane (with the metric normalised in such a way that the sectional curvature lies in $[\frac14, 1]$), consider a geodesic sphere of radius $r \in (0, \pi)$. It is well known (or can be easily verified by explicitly solving the Jacobi equations) that its principal curvatures are $\frac12 \cot \frac12 r$, with multiplicity $7$, and $\cot r$, with multiplicity $8$. The principal subspaces are the eigenspaces of the Jacobi operator $\tR_\xi$, where $\xi$ is a unit normal vector, corresponding to the eigenvalues $\frac 14$ and $1$, respectively. From Gauss equations we obtain (see e.g., \eqref{eq:Gauss} below) that the geodesic sphere is Einstein if and only if $1+\cot^2 r - (7\cot r + 4 \cot \frac12 r) \cot r = \frac14 + \frac14 \cot^2 \frac12r - (7\cot r + 4 \cot \frac12 r) \cot \frac12 r$. Solving this equation we find $r=r_0$, where $\cot r_0 = -\frac{5 \sqrt{6}}{24}$. 
\end{example}

We prove that this is the only Einstein hypersurface. 

\begin{theorem} \label{t:cayley}
  There are no Einstein hypersurfaces in the Cayley hyperbolic plane. A connected hypersurface in the Cayley projective plane is Einstein if and only if it is a domain of the geodesic sphere of radius $r_0$, as in Example~\ref{ex:sphereO}. 
\end{theorem}

Note that this hypersurface is not a harmonic manifold (for example, because it is not $2$-stein, or because a compact, simply-connected harmonic manifold of an odd dimension must have constant curvature by \cite{Sz}, but it does not). For further study of the properties of the Cayley projective plane and the Cayley hyperbolic plane viewed as harmonic spaces we refer to \cite{EPS}.

\medskip

A harmonic manifold satisfies an infinite sequence of conditions on the curvature tensor and its covariant derivatives, the \emph{Ledger formulas}. The first two of them mean that a harmonic manifold is $2$-stein. Recall that a Riemannian manifold $M$ is called \emph{$2$-stein}, if there exist $c_1, \, c_2 \in \br$ such that for every $x \in M$ and every $X \in T_xM$, we have $\Tr R_X = c_1 \|X\|^2$ and $\Tr (R_X^2) = c_2 \|X\|^4$, where $R_X: T_xM \to T_xM$ is the Jacobi operator (for further properties of $2$-stein manifolds see \cite{N3, NP}).

We prove the following.

\begin{theorem} \label{t:2st}
Suppose $\tM$ is a $2$-stein Riemannian manifold of dimension $n > 2$ and $M \subset \tM$ is a $2$-stein hypersurface.
\begin{enumerate}[label=\emph{(\arabic*)},ref=\arabic*]
  \item \label{it:2sttg}
  If $M$ is totally geodesic, then both $\tM$ and $M$ are of constant curvature.

  \item \label{it:2stcc}
  If $\tM$ is of constant curvature, then $M$ is of the same constant curvature.
\end{enumerate}
\end{theorem}
The hypersurfaces in case~\eqref{it:2stcc} are known by different names in the literature (hypersurfaces of conullity $1$, strongly $(n-2)$-parabolic hypersurfaces, developable hypersurfaces) and are very well understood. It is known that a hypersurface in a space of constant curvature of dimension $n$ has the same constant curvature if and only if it is locally foliated by totally geodesic submanifolds of dimension $n-2$ and the normal vector is parallel along the leaves if and only if the rank of the Gauss map (Euclidean, spherical or hyperbolic respectively) at every point is at most $1$.


\section{Preliminaries}
\label{s:pre}

\subsection{Totally geodesic submanifolds}
\label{ss:totgeod}

Let $\tM$ be a Riemannian manifold and let $\on$ and $\tR$ be the Levi-Civita connection and the curvature tensor of $\tM$ respectively. Let $x \in \tM$. A subspace $L \subset T_x\tM$ is called \emph{$\tR$-invariant}, if $\tR(L,L)L \subset L$. This property is equivalent to the fact that for any $T \in L$, the subspace $L$ is invariant relative to the Jacobi operator $\tR_T$ defined by $\tR_X(Y)=\tR(Y,X)X$ for all $X,Y\in T_x\tM$, and is equivalent to the fact that for any $T \in L$, the subspace $L$ is spanned by eigenvectors of $\tR_T$.

If $M$ is a totally geodesic submanifold of $\tM$, then for any $x \in M$, the subspace $T_xM \subset T_x\tM$ is $\tR$-invariant. Moreover, the subspace $T_xM$ is also $(\on \, \tR)$-invariant, that is, for any $T_1, T_2, T_3, T_4 \in T_xM$ we have $(\on_{T_1} \tR)(T_2, T_3) T_4 \in T_xM$. Note that the $\tR$-invariance property, if it is satisfied locally, is also sufficient for total geodesicity, in the following sense. Let $x \in \tM$ and let $L \subset T_x\tM$ be a subspace. For $T \in T_x\tM$ and (a small) $t \in \br$, denote $L_{tT} \subset T_{\exp(tT)}\tM$ the subspace obtained by the parallel translation of $L$ along the geodesic $t \mapsto \exp(tT)$ to the point $\exp(tT)$. By a result of \'{E}.\,Cartan \cite{Car}, there exists a totally geodesic submanifold $M \subset \tM$ passing through $x$ such that $T_xM = L$ if and only if for a some $\varepsilon > 0$ the subspaces $L_{tT}$ are $\tR$-invariant, for all unit vectors $T \in L$ and all $t \in (-\varepsilon, \varepsilon)$. Note that when the latter condition is satisfied, $M$ is (locally) the union of geodesic segments of $\tM$ passing through $x$ and tangent to $L$.

\subsection{Damek-Ricci spaces}
\label{ss:DRspaces}

Let $(\mathfrak{n}, \ip)$ be a metric, two-step nilpotent Lie algebra with the centre $\gz$ and with $\gv=\gz^\perp$. For $Z \in \gz$, define $J_Z \in \mathfrak{so}(\gv)$ by $\<J_ZU,V\>=\<[U,V],Z\>$ for $U,V \in \gv$. The metric algebra $(\mathfrak{n}, \ip)$ is called a \emph{generalised Heisenberg algebra} if for all $Z \in \gz$ we have $J_Z^2=-\|Z\|^2 \id_{\gv}$. Note that $\gv$ is a Clifford module over the Clifford algebra $\mathrm{Cl}(\gz, -\ip_{\gz})$. Consider a one-dimensional extension $\gs=\mathfrak{n} \oplus \ga$ of a generalised Heisenberg algebra $\mathfrak{n}$, where $\ga = \br A$ and $[A,U]=\frac12 U, \; [A,Z]=Z$ for $U \in \gv, \; Z \in \gz$ and extend the inner product from $\mathfrak{n}$ to $\gs$ in such a way that $A \perp \mathfrak{n}$ and $\|A\|=1$. Then $\gs$ is a metric, solvable Lie algebra. The corresponding simply connected Lie group $S$ with the left-invariant metric defined by $\ip$ is called a \emph{Damek-Ricci space}.

Let $T_1, T_2 \in T_eS=\gs$, with $T_1=V+Y+sA, \; T_2 = U+X+rA$, where $V,U \in \gv, \, Y, X \in \gz$. Then according to \cite[\S 4.1.8]{BTV}, for the Jacobi operator of $S$ at $e$ and its covariant derivative we have
\begin{align}\label{eq:Jac}
  &
  \begin{aligned}
    \tR_{T_1}T_2 & = \tfrac34 J_XJ_YV + \tfrac34 J_{[U,V]}V + \tfrac34 r J_YV -\tfrac34 s J_XV -\tfrac14 \|T_1\|^2 U + (\tfrac34 \<X, Y\> + \tfrac14 \<T_1,T_2\>)V \\
     & -\tfrac34 [U, J_YV] + \tfrac34 s [U,V] - (\|T_1\|^2 - \tfrac34 \|V\|^2) X + \<T_1,T_2\>Y \\
     & +(\tfrac34 \<U,J_YV\> - r(\|T_1\|^2 - \tfrac34 \|V\|^2) + s (\<T_1,T_2\> - \tfrac34 \<U,V\>))A,
  \end{aligned}
  \\
  &(\on_{T_1}\tR_{T_1})T_2 = \tfrac32(J_{[U,V]}J_YV + J_{[U,J_YV]}V - \<U,V\>J_YV - \<U,J_YV\>V). \label{eq:nablaR}
\end{align}

By \cite[\S 4.1.7]{BTV}, the sectional curvature of $S$ at $e$ with respect to the two-plane $\sigma$ spanned by orthonormal vectors $V+Y+sA$ and $U+X$ is given by
\begin{equation}\label{eq:secDR}
k(\sigma)=-\tfrac34\|sX-[U,V]\|^2-\tfrac34\<X,Y\>^2-\tfrac14(3\|X\|^2\|Y\|^2+6\<J_XU,J_YV\>+1).
\end{equation}

Throughout the proof, we will use the following identities:
\begin{equation} \label{eq:brackets}
  [V,J_YV] = \|V\|^2 Y, \quad [V,J_YU] - [J_YV,U] = 2 \|Y\|^2 \<U, V\>, \qquad \text{for } U, V \in \gv, \, Y \in \gz.
\end{equation}

Following \cite[\S 3.1.12]{BTV}, for nonzero vectors $V \in \gv$ and $Y \in \gz$ we define the operator $K_{V,Y}$ on the subspace $Y^\perp \cap \gz$ by $K_{V,Y}X =  \|V\|^{-2}\|Y\|^{-1} [V, J_XJ_YV]$. The operator $K_{V,Y}$ is skew-symmetric, with all the eigenvalues of $K_{V,Y}^2$ lying in $[-1, 0]$. Furthermore,
\begin{equation}\label{eq:K2-1}
  K_{V,Y}^2X = - X \; \Leftrightarrow \; J_XJ_YV=\|Y\|J_{K_{V,Y}X}V.
\end{equation}

\section{Totally geodesic submanifolds of Damek-Ricci spaces} 
\label{s:tgDR}

\begin{proof}[Proof of Theorem~\ref{t:tgDR}] Let $M$ be a totally geodesic submanifold of a Damek-Ricci space $S$. Without loss of generality we can assume that $M$ passes through the identity element $e$; denote $L=T_eM$. Furthermore, without loss of generality, we will assume that $A \not\in L^\perp$. Indeed, choosing an arbitrary unit vector $T \in L$, we can parallelly translate $L$ along a small interval of the geodesic in the direction of $T$ and then move it back to pass through the identity by the left translation in $S$. Then by \cite[Theorem~2, \S 4.1.11]{BTV}, the vector $T$ maps to a vector which is not orthogonal to $A$.

By \cite[Theorem~4.2]{BTV}, the eigenvalues of the Jacobi operator of a unit vector tangent to a Damek-Ricci space belong to $[-1, 0]$. We first consider the generic case.

{
\begin{lemma} \label{l:generic}
Suppose that there exists a unit vector $T \in L$ such that the restriction of $\tR_T$ to $L \cap T^\perp$ has an eigenvalue $\kappa$ different from $-\frac14$ and $-1$. Then there exist nonzero vectors $X \in \gz$ and $V \in \gv$ such that $X, V, J_XV \in L$.
\end{lemma}
\begin{proof}
Let $T=V+Y+sA, \, V \in \gv, \, Y \in \gz$. The fact that the restriction of $\tR_T$ to $T^\perp$ has an eigenvalue $\kappa \not\in \{-\frac14, -1\}$ means that $T$ is as in case~(vi)(B) of \cite[Theorem~4.2]{BTV}. Then necessarily $V \ne 0$ and $Y \ne 0$. Moreover, as our condition is open we can assume that $s \ne 0$. The corresponding eigenvector $E \in L$ is constructed as follows. Consider the operator $K=K_{V,Y}$ and take an eigenvalue $\mu \ne -1$ of $K^2$ and the corresponding unit eigenvector $X$. Then $E \in L$ is given by
\begin{multline} \label{eq:eigen}
  E=(4\kappa + 1)(4\kappa + 1 + 3 \|V\|^2)X + 3 (4\kappa + 1 + 3 \|V\|^2) J_XJ_YV \\ - 3s (4\kappa + 1) J_XV - 9\|V\|^2 \|Y\| J_{KX}V,
\end{multline}
where $\kappa$ is a root of the cubic equation
\begin{equation} \label{eq:cubic}
(4\kappa+4)(4\kappa+1)^2 = 27 \|V\|^4 \|Y\|^2 (1+\mu).
\end{equation}
Note that $4\kappa + 1 + 3 \|V\|^2 \ne 0$ as otherwise from \eqref{eq:cubic} we would obtain $1-\|V\|^2=\|Y\|^2 (1+\mu)$. As $\mu \le 0$ and $T$ is a unit vector, this implies $s=0$ which contradicts our choice of $T$.

As $L$ must be $(\on \, \tR)$-invariant we have $(\on_T \tR_T)E, (\on_T \tR_T)^2 E \in L$. From \eqref{eq:nablaR} $(\on_T\tR_T)X=0$. Substitute each of the vectors $J_XV, J_XJ_YV, J_{KX}V$ and $J_{KX}J_YV$ for $U$ in \eqref{eq:nablaR}. As $X, KX \perp Y$, it is easy to see that $J_XV, J_XJ_YV, J_{KX}V, J_{KX}J_YV \perp V, J_YV$. Furthermore, from \eqref{eq:brackets}, by definition of $K$ and from the fact that $K^2X = \mu X$ we get
\begin{gather*}
  [J_XV,V] = -\|V\|^2X, \; [J_{KX}V,V] = -\|V\|^2KX, \; [J_XJ_YV,J_YV]=-\|V\|^2\|Y\|^2 X, \\
  [J_{KX}J_YV,J_YV]=-\|V\|^2\|Y\|^2 KX, \; [J_XV,J_YV] = [V,J_XJ_YV] = \|V\|^2\|Y\| KX, \\
  [J_{KX}V,J_YV] = [V,J_{KX}J_YV] = \|V\|^2\|Y\| K^2X = \mu \|V\|^2\|Y\|.
\end{gather*}
Then \eqref{eq:nablaR} gives
\begin{equation}\label{eq:nablaTRT}
\begin{split}
\tfrac23(\on_T\tR_T) J_XV & = -\|V\|^2 J_XJ_YV + \|V\|^2 \|Y\| J_{KX}V, \\
\tfrac23(\on_T\tR_T) J_{KX}J_YV & = -\|V\|^2 \|Y\| \mu J_XJ_YV - \|V\|^2 \|Y\|^2 J_{KX}V, \\
\tfrac23(\on_T\tR_T) J_XJ_YV & = - \|V\|^2 \|Y\|^2 J_{X}V -\|V\|^2 \|Y\| J_{KX}J_YV, \\
\tfrac23(\on_T\tR_T) J_{KX}V & =  \|V\|^2 \|Y\| \mu J_{X}V-\|V\|^2 J_{KX}J_YV.
\end{split}
\end{equation}

Consider two cases.

First assume that $\mu=0$. Then $K^2X=0$, and hence $KX=0$. From \eqref{eq:nablaTRT} we obtain $\tfrac23(\on_T\tR_T) J_XV = -\|V\|^2 J_XJ_YV, \; \tfrac23(\on_T\tR_T) J_XJ_YV = - \|V\|^2 \|Y\|^2 J_{X}V$. As $(\on_T\tR_T) X = 0$, the fact that $E, (\on_T \tR_T)E, (\on_T \tR_T)^2 E \in L$ implies that $X, J_XV \in L$. But then considering the eigenspaces of $\tR_X$ by case (ii) of \cite[Theorem~4.2]{BTV} we obtain that $L$ contains the projections of every its vector to both $\gv$ and $\gz \oplus \ga$. In particular, $L$ contains $V$, the projection of $T$ to $\gv$. This completes the proof in the first case.

Now assume that $\mu \ne 0$, so that $\mu \in (-1,0)$ and $KX \ne 0$. From \eqref{eq:nablaTRT} it follows that the subspace $\mathfrak{l}_4=\Span(J_XV, J_XJ_YV, J_{KX}V, J_{KX}J_YV)$ is invariant under $(\on_T\tR_T)$. Note that $\dim \mathfrak{l}_4= 4$. Indeed, as $K$ is skew-symmetric, we have $KX \perp X$ and so $J_XV, J_{KX}J_YV \perp J_XJ_YV, J_{KX}V$. Moreover, $\<J_{KX}V, J_XJ_YV \> = \<KX, [V, J_XJ_YV]\> = \|V\|^2 \|Y\| \|KX\|^2 = -\mu \|V\|^2 \|Y\|$. But then $\<J_XV, J_{KX}J_YV\>^2 = \mu^2 \|V\|^4 \|Y\|^2 < \|J_XV\|^2 \|J_{KX}J_YV\|^2 = \|KX\|^2  \|V\|^4 \|Y\|^2 = -\mu \|V\|^4 \|Y\|^2$, as $\mu \in (-1, 0)$, and so the vectors $J_XV, J_{KX}J_YV$ are linear independent. Acting on them by $J_Y$ we obtain that the vectors $J_XJ_YV, J_{KX}V$ are also linear independent. It follows that the quadruple of vectors $(J_XV, J_{KX}J_YV, J_XJ_YV, J_{KX}V)$ is a basis (in general, not orthonormal) for $\mathfrak{l}_4$. From the above, the matrix $Q$ of the restriction of the operator $\tfrac23(\on_T\tR_T)$ to $\mathfrak{l}_4$ relative to that basis is given by 
\begin{equation*}
  Q = \|V\|^2 \left(
        \begin{array}{cccc}
          0 & 0 & -1 & \|Y\| \\
          0 & 0 & -\mu \|Y\| & -\|Y\|^2 \\
          -\|Y\|^2 & -\|Y\| & 0 & 0 \\
          \mu \|Y\| & -1 & 0 & 0 \\
        \end{array}
      \right).
\end{equation*}
As $Q^2 = \|V\|^4 \|Y\|^2 (1+\mu) I$, the restriction of the operator $\tfrac49 \|V\|^{-4} \|Y\|^{-2} (1+\mu)^{-1}(\on_T\tR_T)^2$ to $\mathfrak{l}_4$ is the identity operator. It now follows from \eqref{eq:eigen} that $(\id - \tfrac49 \|V\|^{-4} \|Y\|^{-2} (1+\mu)^{-1}(\on_T\tR_T)^2)E = (4\kappa + 1)(4\kappa + 1 + 3 \|V\|^2)X$. As $\kappa \ne -\frac14$ by assumption and $4\kappa + 1 + 3 \|V\|^2 \ne 0$ from the above, and as $L$ is $(\on \tR)$-invariant, we obtain that $L$ contains $X$. Similar to the above, by case (ii) of \cite[Theorem~4.2]{BTV} we obtain that $L$ contains the projections of all its vectors to both $\gv$ and $\gz \oplus \ga$. In particular, $V \in L$. Furthermore, $L$ contains the projection of $\tR_TX$ to $\gv$ which by \eqref{eq:Jac} equals $J_XJ_YV-sJ_XV \in L$. But then, as $L$ also contains $E$ and $X$, we find from~\eqref{eq:eigen} that $J_XJ_YV - \|Y\| J_{KX}V \in L$. Then $L$ also contains the vectors $\tfrac23(\on_T\tR_T)(J_XJ_YV-sJ_XV)$ and $\tfrac23(\on_T\tR_T)(J_XJ_YV-\|Y\|J_{KX}V)$. Relative to the basis $(J_XV, J_{KX}J_YV, J_XJ_YV, J_{KX}V)$, the coordinate vectors of $J_XJ_YV-sJ_XV$ and of $J_XJ_YV-\|Y\|J_{KX}V$ are $a=(-s,0,1,0)^t$ and $b=(0,0,1,-\|Y\|)^t$ respectively. Then we have $Qa=(-1,-\mu \|Y\|, s \|Y\|^2, -s \mu \|Y\|)^t, \; Qb=(-1-\|Y\|^2, -\mu \|Y\| + \|Y\|^3,0,0)^t$ and the determinant of the $4 \times 4$-matrix whose vector columns are $a, b, Qa, Qb$ equals $s^2\|Y\|^2(\mu-\|Y\|^2)^2 + \|Y\|^4 (1+\mu)>0$. This implies that $L \supset \mathfrak{l}_4$, and in particular, $J_XV \in L$.
\end{proof}
}

Now, suppose $X, V, J_XV \in L$, for some nonzero $X \in \gz, \, V \in \gv$. Then by \cite[Theorem~4.2(ii)]{BTV}, $L$ contains the projection of any of its vectors to $\ga \oplus \gz$, so in particular, the projection of the vector $\tR_{V+X}J_XV$ to $\ga \oplus \gz$ which by \eqref{eq:Jac} equals $\frac34\|X\|^2\|V\|^2 A$. It follows that $A \in L$. To complete the proof in this (generic) case we need the following lemma.

{
\begin{lemma}\label{l:tgdrA}
Suppose $M$ is a totally geodesic submanifold of a Damek-Ricci space $S$ and $e \in M$. If $A \in L=T_eM$, then $M$ is a homogeneous totally geodesic submanifold.
\end{lemma}
\begin{proof}
Suppose that $L$ contains a nonzero vector $U+X, \; U \in \gv, X \in \gz$. Then by \eqref{eq:Jac} we get $\tR_A (U+X) = -\frac14 U - X$. As $L$ is $\tR$-invariant, we get $U, X \in L$, and so $L = \br A \oplus \gv' \oplus \gz'$, where $\gv'$ and $\gz'$ are subspaces of $\gv$ and $\gz$ respectively. Next, by \eqref{eq:Jac} for any $V+Y+sA \in L$, we have $\tR_{V+Y+sA}A=\frac34 J_YV + \frac14 sV + sY - (\frac14 \|V\|^2 + \|Y\|^2) A$. As $V, Y, A \in L$ we deduce that $J_YV \in L$, for all $V \in \gv'$ and $Y \in \gz'$. Furthermore, for any $U, V \in \gv'$, the subspace $L$ contains the projection of the vector $\tR_{V+A}U$ to $\gz$ which by \eqref{eq:Jac} equals $\frac34 [U,V]$. It follows that $L$ is a subalgebra. As a totally geodesic submanifold is (locally) uniquely determined by its tangent space at a point, the claim follows from Theorem~\ref{t:tgDRhomo}.
\end{proof}
}

To complete the proof of the theorem it remains to consider the case when for any unit vector $T \in L$, any eigenvalue of the restriction of the Jacobi operator $\tR_T$ to $L \cap T^\perp$ is either $-\frac14$ or $-1$. We note in passing that by continuity, the multiplicities of the eigenvalues $-\frac14$ and $-1$ must be constant, and so $M$ must be an Osserman manifold. This \emph{almost} completes the proof, as by \cite{N1, N2}, any Osserman manifold of dimension different from $16$ is flat or rank-one symmetric.

We will first show that $M$ is locally symmetric. Let $T = V + Y + sA \in L, \; V \in \gv, \, Y \in \gz$, be a unit vector. If $Y=0$ or $V=0$, equation~\eqref{eq:nablaR} immediately implies that $(\on_{T} \tR_{T})T' = 0$, for any $T' \in L$. Otherwise, suppose that $Y \ne 0$ and $V \ne 0$. Then by cases~(vi)(1), (2) and (3A) of \cite[Theorem~4.2]{BTV}, any eigenvector $T'$ of the restriction of $\tR_T$ to $L \cap T^\perp$ is a linear combination of vectors whose $\gv$-components belong to the space $\gv'=\Span(V, J_YV) \oplus \{W \in \gv \, | \, [W,V]=[W,J_YV]=0\} \oplus \{J_{\|Y\|KX-sX}V \, | \, X \in \gz \cap Y^\perp, K^2X = -X\}$ (in the latter subspace, we denoted $K=K_{V,Y}$ and used \eqref{eq:K2-1}). By equation~\eqref{eq:nablaR}, $(\on_{T} \tR_{T})T'$ only depends on the $\gv$-component of $T'$ and so to prove that $(\on_{T} \tR_{T})T'=0$ it suffices to show that $(\on_{T} \tR_{T})U= 0$, for all $U \in \gv'$. From~\eqref{eq:nablaR}, this fact is immediate for $U$ satisfying $[U,V]=[U,J_YV]=0$, and also easily follows from~\eqref{eq:brackets} for $U= V$ and $U = J_YV$. If $U=J_{\|Y\|KX-sX}V$, where $X \in \gz \cap Y^\perp, \, K^2X = -X$, we have $\<U,V\>=\<U,J_YV\>=0, \; [U,V]=-\|V\|^2(\|Y\|KX-sX)$ and $[U,J_YV]=[V,J_{\|Y\|KX-sX}J_YV]=\|V\|^2\|Y\|K(\|Y\|KX-sX)=\|V\|^2\|Y\|(-\|Y\|X-sKX)$ by~\eqref{eq:brackets}. It follows that $\tfrac23(\on_{T}\tR_{T})U =  -\|V\|^2 J_{\|Y\|KX-sX}J_YV-\|V\|^2\|Y\| J_{\|Y\|X+sKX}V=0$, since $J_XJ_YV=\|Y\|J_{KX}V$ and $J_{KX}J_YV=\|Y\|J_{K^2X}V = -\|Y\|J_{X}V$ by~\eqref{eq:K2-1}. Thus for all $T, T' \in L$ we have $(\on_{T} \tR_{T})T' = 0$. As $M$ is totally geodesic, the same is true if we replace $\on$ and $\tR$ by the Levi-Civita connection and the curvature tensor of the induced metric on $M$ respectively. But then by \cite[Lemma~5.1]{VW}, $M$ is locally symmetric. The fact that $M$ is rank-one symmetric follows from the fact that its sectional curvature lies in $[-1,-\frac14]$.
\end{proof}

Note that there are many rank-one symmetric totally geodesic submanifolds in a general Damek-Ricci space, and that they do not need to be homogeneous (as totally geodesic submanifolds; of course, they are homogeneous by themselves as Riemannian spaces).

\begin{example} \label{ex:hyperwell}
Let $L=\ga \oplus \gz'$ or $L=\ga \oplus \gv'$, where $\gz'$ is an arbitrary subspace of $\gz$, and $\gv'$ is an arbitrary \emph{abelian} subspace of $\gv$. Then $L$ is well-positioned in the sense of Theorem~\ref{t:tgDRhomo} and is tangent to a homogeneous totally geodesic hyperbolic space of curvature $-1$ or $-\frac14$ respectively. Now take any hyperplane $L' \subset L$ which is not well-positioned. Then it is tangent to a non-homogeneous totally geodesic hyperbolic space. Similar examples can be constructed starting with a rank-one homogeneous totally geodesic submanifold of non-constant curvature.
\end{example}

\begin{example} \label{ex:except}
Let $\dim \gz = 6$, and let $\gv$ be the $8$-dimensional irreducible module over the Clifford algebra $\mathrm{Cl}(\gz)$. The corresponding Damek-Ricci space is a non-symmetric space of dimension $15$. Take an orthonormal basis $X_i, \; i=1, \dots, 6$, for $\gz$. We abbreviate $J_{X_i}$ to $J_i$. The operator $J_7:=J_1J_2J_3J_4J_5J_6$ is orthogonal, skew-symmetric and anti-commutes with all the operators $J_i, \, i=1, \dots, 6$. So the operators $J_i, \, i=1, \dots, 6,7$, give a representation of the Clifford algebra $\mathrm{Cl}_7$ on $\gv$. The operator $J_1J_2J_3$ is symmetric; let $W$ be its eigenvector (the corresponding eigenvalue $\varepsilon$ is always $\pm 1$) and let $V = a W + b J_7 W$, for $a, b \in \br$, not both zeros. It is easy to check that $J_1J_2V, J_2J_3V, J_3J_1V \perp J_7V$, and so $J_1J_2V, J_2J_3V, J_3J_1V \in J_\gz V$. Let $Z \in \gz$ be such that $J_1J_2V=J_ZV$. Then $\|Z\|=1$ and from \eqref{eq:brackets} we get $\|[J_1V, J_2V]\|=\|[V, J_1J_2V]\|=\|[V, J_ZV]\|=\|V\|^2$, and similarly, $\|[J_2V, J_3V]\|=\|[J_3V, J_1V]\|=\|V\|^2$. Now let $s \in \br$ and let $T_0=V+sA, \, T_i = sX_i+J_iV, \; i=1,2,3$. From~\eqref{eq:secDR}, the sectional curvature of $S$ with respect to any two-plane $\Span(T_i, T_j), \; 0 \le i < j \le 3$, is $-1$. From the fact that $-1$ is the minimum of the sectional curvature of $S$ it follows that the subspace $L=\Span(T_0, T_1, T_2, T_3)$ is $\tR$-invariant and the sectional curvature of $S$ with respect to any two-plane in $L$ is $-1$ (see a detailed argument in the first paragraph of the proof of Proposition~\ref{p:-1} below). Let $M$ be the submanifold of $S$ obtained by taking all the geodesics passing through $e$ in the directions of vectors from $L$. To show that $M$ is indeed totally geodesic, consider a geodesic $\gamma$ of $S$ such that $\gamma(0)=e$ and $\dot{\gamma}(0) = X$, where $X \in L$ is a unit vector. Let $x = \gamma(t)$ for some $t > 0$ and let $Y' \in T_xS$. It suffices to show that the geodesic of $S$ passing through $x$ in the direction of $Y'$ lies on $M$. Note that $Y'=F(t)$, where $F$ is a Jacobi field of $S$ along $\gamma$ such that $F(0)=0$ and $\dot{F}(0) = Y \in L$. Rotating the triple $(X_1, X_2, X_3)$ if necessary we can assume without loss of generality that $X, Y \perp T_3$. Then $X, Y \in \gs'$, where $\gs'=\Span(A, X_1, X_2, Z, V, J_1V, J_2V, J_ZV)$. Note that $\gs'$ is well-positioned and is the tangent space at $e$ to the totally geodesic $\mathbb{H}H^2 \subset S$. The two-plane $\Span(X,Y)$ is $\tR$-invariant and the union of geodesics of $\mathbb{H}H^2$ passing through $e$ in the directions of vectors from $\Span(X,Y)$ is the hyperbolic plane $H$ of curvature $-1$ which is totally geodesic in $\mathbb{H}H^2$. But then $H$ is totally geodesic in $S$, and so $\gamma$ lies on $H$ and the Jacobi field $F$ of $S$ along $\gamma$ is a Jacobi field of $H$ along $\gamma$, with the same initial conditions. It follows that $x \in H$ and $Y' \in T_xH$, and the geodesic of $H$ passing through $x$ in the direction of $Y'$ is a geodesic of $S$ lying on $M$.
\end{example}

We can be more specific in the case when the totally geodesic submanifold $M \subset S$ is of constant curvature $-1$. As any totally geodesic submanifold of such a submanifold is again totally geodesic in $S$ and is of constant curvature $-1$, it is sufficient to classify only the maximal ones by inclusion.

\begin{proposition} \label{p:-1}
Let $M$ be a connected submanifold of a Damek-Ricci space $S$. Suppose $e \in M$ and denote $L=T_eM$. Then $M$ is maximal, totally geodesic submanifold of sectional curvature $-1$ if and only if one of the following occurs.
\begin{enumerate}[label=\emph{(\arabic*)},ref=\arabic*]
  \item \label{it:-1homo}
  The submanifold $M$ is a homogeneous totally geodesic submanifold and $L=\ga \oplus \gz$.

  \item \label{it:hyper-1homo}
  The submanifold $M$ is a non-homogeneous totally geodesic submanifold. Then $\dim M \in \{2,4,8\}$ and $M$ is a totally geodesic submanifold of sectional curvature $-1$ of a homogeneous totally geodesic Damek-Ricci submanifold $S' \subset S$. If $\dim M = 2$, then $S'$ is isometric to  $\mathbb{C}H^2$; if $\dim M = 8$, then $S'$ is isometric to $\mathbb{O}H^2$; if $\dim M = 4$, then $S'$ is isometric to  either $\mathbb{H}H^2$, or to the Damek-Ricci space of dimension $15$ and the pair $(M, S')$ is as constructed in Example~\ref{ex:except}.
\end{enumerate}

\end{proposition}
\begin{proof}
We call a subspace $\gs' \subset \gs$ such that the sectional curvature of $S$ with respect to any two-plane from $\gs'$ is $-1$ a \emph{$(-1)$-subspace}. As $-1$ is the minimum of the sectional curvature of $S$, we obtain that if for two orthonormal vectors $T_1, T_2 \in \gs$, the sectional curvature of $S$ with respect to $\Span(T_1, T_2)$ is $-1$, then $T_2$ is an eigenvector of $\tR_{T_1}$, with the eigenvalue $-1$. It follows that a $(-1)$-subspace is $\tR$-invariant. Moreover, for any three linear independent vectors $T_1, T_2, T_3 \in \gs$, if the curvature of $S$ with respect to both $\Span(T_1, T_2)$ and $\Span(T_1, T_3)$ is $-1$, then the curvature with respect to $\Span(T_1, T)$ is also $-1$, for any nonzero $T \in \Span(T_2, T_3)$. Hence a subspace $\gs' \subset \gs$ is a $(-1)$-subspace if and only if it has a basis $T_1, \dots, T_m$ such that the curvature of $S$ with respect to $\Span(T_i, T_j)$ is $-1$, for any $1 \le i < j \le m$.

By assumption, the subspace $L=T_eM$ is a $(-1)$-subspace. By the argument in the first paragraph of the proof of Theorem~\ref{t:tgDR}, we can always assume that $L$ is not orthogonal to $A$. We will show that any such subspace $L$ which is maximal (by inclusion) is tangent to one of the totally geodesic submanifolds listed in the statement of the proposition.

Denote $L' = L \cap A^\perp$ and let $T=V+Y+sA \in L$ be a unit vector orthogonal to $L'$. Then $s \ne 0$. If $V=0$, then by \cite[Theorem~4.2(i, iv)]{BTV} $L \subset \ga \oplus \gz$. As $L$ is maximal and as $\ga \oplus \gz$ is a $(-1)$-subspace tangent to a homogeneous totally geodesic hyperbolic space of curvature $-1$, we must have $L= \ga \oplus \gz$ which gives case~\eqref{it:-1homo}.

Now suppose $V \ne 0$. We first show that $Y =0$. Assuming $Y \ne 0$ we find from \cite[Theorem~4.2(vi)]{BTV} that any $T' \in L'$ has the form $T'=-(\|Y\|^2+s^2)X+J_X(J_YV-sV)$, where $X \in \gz \cap Y^\perp$ is a nonzero vector such that $K_{V,Y}^2X=-X$. But then the vector $T_1=J_YV+sY-\|Y\|^2A$ does not lie in $L$ and we have $k(\Span(T,T_1))=k(\Span(T',T_1))=-1$ from~\eqref{eq:secDR}, so that $\br T_1 \oplus L$ is also a $(-1)$-subspace contradicting the fact that $L$ is maximal.

It follows that $T=V+sA, \, V \ne 0, s \ne 0$, and then by \cite[Theorem~4.2(v)]{BTV}, there is a subspace $\gz' \subset \gz$ such that $L' = \{sX+J_XV \, | \, X \in \gz'\}$. Let $X_1, X_2 \in \gz'$ be orthonormal. Then the vectors $T_1= sX_1+J_{X_1}V, \; T_2= sX_2+J_{X_2}V \in L'$ are also orthonormal and from~\eqref{eq:secDR} we obtain $k(\Span(T_1,T_2))=-1 + \frac34(\|V\|^4-\|[J_{X_1}V,J_{X_2}V]\|^2)$. It follows that $\|[J_{X_1}V,J_{X_2}V]\|^2=\|V\|^4$, and so by~\eqref{eq:brackets}, $\|K_{V,X_2}X_1\|^2=1$. As $K_{V,X_2}$ is a skew-symmetric operator, with all the eigenvalues of $K_{V,X_2}^2$ lying in $[-1,0]$ we find that $K_{V,X_2}^2X_1=-X_1$, and so by~\eqref{eq:K2-1}, $J_{X_1}J_{X_2}V \in \cJ V$.

Thus $L$ is a $(-1)$-subspace if and only if
\begin{equation}\label{eq:weakJ2}
J_{X_1}J_{X_2}V \in \cJ V \oplus \br V, \quad \text{for any } X_1, X_2 \in \gz'.
\end{equation}
This property is weaker than the $J^2$-property, but is still very restrictive.

Denote $\gz_0$ the linear span of all vectors $Z \in \gz$ with the property that $J_{X_1}J_{X_2}V=J_ZV$, for some orthogonal vectors $X_1, X_2 \in \gz'$. Denote $\gz''=\gz'+\gz_0$ and $d=\dim \gz'$. Let $\gv'=\Span(V,J_{X_1}\dots J_{X_k}V \, | \, k \ge 1, \, X_1, \dots, X_k \in \gz')$. Note that $\gv'$ is a $\mathrm{Cl}(\gz')$-module, where $\mathrm{Cl}(\gz')$ is the Clifford algebra over $(\gz', -\ip)$. Moreover, as $\gv'$ is generated by a single vector, its decomposition into irreducible $\mathrm{Cl}(\gz')$-modules cannot contain two isomorphic modules. It follows that $\gv'$ is either an irreducible $\mathrm{Cl}(\gz')$-module, or is the direct some of two non-isomorphic irreducible modules (this may only occur when $d \equiv 3 \pmod 4$).

{
\begin{lemma} \label{l:weakJ2}
  In the above notation,
  \begin{enumerate}[label=\emph{(\alph*)},ref=\alph*]
    \item \label{it:gv'mod}
    The module $\gv'$ is a $\mathrm{Cl}(\gz'')$-module.

    \item \label{it:gz'dim}
    We have $\dim \gz'' \le 7$ and one of the following can occur. In all the cases, $\gv'$ is an irreducible $\mathrm{Cl}(\gz'')$-module.
    \begin{enumerate}[label=\emph{(\roman*)},ref=\roman*]
      \item \label{it:11}
      $\dim \gz' = \dim \gz'' = \dim \gv' = 1$;

      \item \label{it:233}
      $\dim \gz' \in \{2, 3\}, \, \dim \gz'' = 3$ and $\dim \gv'=4$.

      \item \label{it:36}
      $\dim \gz' = 3, \, \dim \gz'' = 6$, and $\dim \gv'=8$.

      \item \label{it:477}
      $\dim \gz' \in \{4, 5, 6, 7\}, \, \dim \gz'' = 7$ and $\dim \gv'=8$.
    \end{enumerate}
  \end{enumerate}
\end{lemma}
\begin{proof}
\eqref{it:gv'mod} Let $Z \in \gz$ be such that $J_{X}J_{X'}V=J_ZV$ for some orthogonal vectors $X, X' \in \gz'$. Then for any $X_1, \dots, X_k \in \gz'$ we have $J_ZJ_{X_1}\dots J_{X_k}V - (-1)^kJ_{X_1}\dots J_{X_k}J_ZV \in \gv'$ and $J_{X_1}\dots J_{X_k}J_ZV=J_{X_1}\dots J_{X_k}J_XJ_{X'}V \in \gv'$.

\eqref{it:gz'dim} We first show that $d \le 7$. Suppose that $d \ge 8$ and take two orthonormal vectors $X_1, X_2 \in \gz'$. Then $J_{X_1}J_{X_2}V=J_{Z_1}V$ for some $Z_1 \in \gz$. It is easy to see that $X_1, X_2$ and $Z_1$ are orthonormal. Take an arbitrary unit vector $Y_1 \in \gz'$ such that $Y_1 \perp X_1, X_2, Z_1$ and denote $\gz_1 = \gz' \cap \Span(X_1, X_2, Z_1, Y_1)^\perp$. Note that $\dim \gz_1 \ge 4$, and so the three linear forms $Y' \mapsto \<J_{Y_1}J_{Y'}V,J_{X_1}V\>, \;  Y' \mapsto \<J_{Y_1}J_{Y'}V,J_{X_2}V\>$ and $Y' \mapsto \<J_{Y_1}J_{Y'}V,J_{Z_1}V\>$ have a nontrivial common kernel on $\gz_1$. Let $Y_2$ be a unit vector in that kernel and let $Z_2 \in \gz$ be such that $J_{Y_1}J_{Y_2}V=J_{Z_2}V$. Note that by construction, the six vectors $X_1, X_2, Z_1, Y_1, Y_2$ and $Z_2$ are orthonormal. Moreover, we have $J_{X_1}J_{X_2}J_{Z_1}V=J_{Y_1}J_{Y_2}J_{Z_2}V=-V$, and so $J_{X_1}J_{X_2}J_{Z_1}J_{Y_1}J_{Y_2}J_{Z_2}V=V$. But the operator $J_{X_1}J_{X_2}J_{Z_1}J_{Y_1}J_{Y_2}J_{Z_2}$ is skew-symmetric which gives a contradiction.

We next show that if $d=7$, then $\dim \gz'' \ne 8$. Indeed, suppose that $d=7$ and $\dim \gz'' \ne 8$, and let $Z \in \gz''$ be a unit vector orthogonal to $\gz'$. Then there exist $X_1, X_2 \in \gz'$ such that $\<J_{X_1}J_{X_2}V, J_{Z}V\> \ne 0$. As this condition is open, there exists an open set of pairs $(X_1, X_2)$ with that property. For any such pair, there exists $X_3 \in \gz'$ and $c \ne 0$ such that $J_{X_1}J_{X_2}V = J_{X_3}V + c J_{Z}V$, and so $J_{X_1}J_ZV \in \gz''$, for an open subset of vectors $X_1 \in \gz'$. It follows that the $8$-dimensional space $\gz''$ by itself satisfies the condition~\eqref{eq:weakJ2} which then leads to a contradiction by replacing $\gz'$ by $\gz''$ in the argument in the above paragraph.

We now consider several cases. For dimensions of irreducible Clifford modules we refer the reader to \cite[Table~2]{ABS}.

If $d=1$ we are in case~\eqref{it:11}.

If $d=2$ and $X_1, X_2$ are orthonormal vectors in $\gz'$, then $\gz''=\Span(X_1, X_2,Z)$, where $J_{X_1}J_{X_2}V = J_ZV$ and $\dim \gv'=4$. If $d=3$ and $J_{X_1}J_{X_2}V = J_{X_3}V$ for an orthonormal basis $X_1, X_2, X_3$ for $\gz'$, then $\gz''=\gz'$. The module $\gv'$ is a $\mathrm{Cl}(\gz'')$-module and it cannot be the sum of two non-isomorphic submodules, as the $J^2$-property is satisfied. Hence $\dim \gv'=4$ and we are in case~\eqref{it:233}.

If $d=3$ and $J_{X_1}J_{X_2}V = \pm J_{X_3}V$ for no orthonormal basis $X_1, X_2, X_3$ for $\gz'$, denote $Z_i, \, i=1,2,3$, the unit vectors in $\gz$ defined by $J_{X_i}J_{X_j}V = J_{Z_k}V$, where $(i,j,k)$ is a cyclic permutation of $(1,2,3)$. Then we have $Z_i \perp X_j$ for $i \ne j$, and $J_{Z_i}J_{Z_j}V=J_{Z_k}V$, where $(i,j,k)$ is a cyclic permutation of $(1,2,3)$; in particular, $Z_1, Z_2$ and $Z_3$ are orthonormal. We also have $\<Z_i,X_i\>=-\<J_{X_1}J_{X_2}J_{X_3}V,V\>$, for all $i=1,2,3$. It follows that $\gz''=\gz' \oplus \Span(Z_1, Z_2, Z_3)$, so that $\dim \gz'' = 6$ and then $\gv'$ must be an irreducible $\mathrm{Cl}(\gz'')$-module and $\dim \gv' = 8$, as in case~\eqref{it:36}.

Suppose $4 \le d \le 7$. We can always find an orthonormal basis $X_i, \, i=1, \dots, d$, for $\gz'$ such that $J_{X_1}J_{X_2}V \ne \pm J_{X_3}V$ and then construct the vectors $Z_1, Z_2, Z_3 \in \gz''$ from $X_1, X_2, X_3$ as above. Then $\gz''$ contain the six-dimensional space $\gz_6 = \gz' \oplus \Span(Z_1, Z_2, Z_3)$. Now if $X_4 \not \in \gz_6$, then $\dim \gz'' \ge 7$. If $X_4 \in \gz_6$, then we can rotate the triple $X_1, X_2, X_3$ to get $X_4 \in \Span (X_1, Z_1)$, and then for the vector $Z \in \gz''$ defined by $J_{X_1}J_{X_4}V = J_ZV$ we obtain $J_ZV \in \Span(J_{X_1}J_{X_2}J_{X_3}V, V)$. It follows that $J_ZV \perp J_{\gz_6}V$, hence $Z \perp \gz_6$, and so, again, $\dim \gz'' \ge 7$. To see that $\dim \gz'' = 7$ we first consider the cases $d=4,5,6$. Then $\gv'$ is an irreducible $\mathrm{Cl}(\gz')$-module, and so $\dim \gv' = 8$. But as $\gv'$ is also a $\mathrm{Cl}(\gz'')$-module by \eqref{it:gv'mod} we find that $\dim \gz''=7$ which gives case~\eqref{it:477}. The only remaining case to consider is $d=7$. If $\dim \gz'' = 7$, then the $J^2$-property is satisfied, and so $\gv'$ is an irreducible $\mathrm{Cl}(\gz'')$-module and $\dim \gv'=8$, as required. If $\dim \gz'' \ge 8$, then $\dim \gz'' \ge 9$ as we showed above, and so $\dim \gv' \ge 32$, as $\gv'$ is a $\mathrm{Cl}(\gz'')$-module by \eqref{it:gv'mod}. But on the other hand, the maximal possible dimension of $\gv'$ viewed as a $\mathrm{Cl}(\gz')$-module is $16$ and it is attained when $\gv'$ is the sum of two non-isomorphic irreducible $8$-dimensional $\mathrm{Cl}(\gz')$-modules. This contradiction completes the proof of the lemma.
\end{proof}
}

We now use the condition that the subspace $L = \br(V+sA) \oplus \{sX+J_XV \, | \, X \in \gz'\}$ is a maximal $(-1)$-subspace. From Lemma~\ref{l:weakJ2} we see that in the cases $d=2,4,5,6$ we can simply replace $\gz'$ by $\gz''$ in the definition of $L$ to obtain a bigger $(-1)$-subspace. In the other cases, $L$ \emph{can} be maximal, for example, if $\gz=\gz''$. Note however that if $\gz$ is large, it may happen that $L$ lies in a bigger $(-1)$-subspace (which is still of the form given above, with $\gz', \gz''$ and $\gv'$ as in Lemma~\ref{l:weakJ2}) -- for example, if $d=1$ and $\dim \gz \ge 3$, or if $d=3$ and $\dim \gz'' = 6$ (as in case~\eqref{it:36}) and $\dim \gz \ge 7$.

To summarise, any maximal $(-1)$-subspace has the form $L = \br(V+sA) \oplus \{sX+J_XV \, | \, X \in \gz'\}$ and one of the following may occur.

\begin{itemize}
  \item
  Either $\dim \gz' \in \{1,3,7\}$ and $\gz'$ satisfies the $J^2$-property on $\gv'=\br V \oplus J_{\gz'}V$ (so that for any orthogonal $X_1, X_2 \in \gz'$ there exists $X_3 \in \gz'$ such that $J_{X_1}J_{X_2}V=J_{X_3}V$; it is easy to see that we can replace $V$ by any vector from $\gv'$).

  In these cases, $L$ is a subspace of a well-positioned subalgebra $\gs'=\ga \oplus \gv' \oplus \gz' \subset \gs$ tangent to a homogeneous totally geodesic rank-one submanifold $S' \subset S$ which is isometric to $\mathbb{C}H^2, \mathbb{H}H^2$ or $\mathbb{O}H^2$ respectively. As $L$ is $\tR$-invariant with respect to $S$, it is also $\tR$-invariant with respect to $S'$. But as $S'$ a symmetric space, any $\tR$-invariant subspace of $\gs'$ is tangent to a totally geodesic submanifold $M$ (of $S'$, and hence of $S$). As $L$ is a $(-1)$-subspace, $M$ must be of constant curvature $-1$.

  \item
  Or  $\dim \gz' =3, \, \dim \gz''=6$, and for any orthogonal $X_1, X_2 \in \gz'$ there exists $X_3 \in \gz''$ such that $J_{X_1}J_{X_2}V=J_{X_3}V$. Then $L$ is a subspace of a well-positioned subalgebra $\gs'=\ga \oplus \gv' \oplus \gz'' \subset \gs$, where $\gv'$ is the irreducible $8$-dimensional $\mathrm{Cl}(\gz'')$-module. The algebra $\gs'$ is of dimension $15$ and is tangent to a homogeneous totally geodesic non-symmetric submanifold $S' \subset S$. Then $L$ is indeed the tangent space to a totally geodesic submanifold $M \subset S' \subset S$ of curvature $-1$, as explained in Example~\ref{ex:except}. \qedhere
\end{itemize}
\end{proof}

\section{Einstein hypersurfaces in $\mathbb{O}P^2$ or $\mathbb{O}H^2$}
\label{s:EinC}


\begin{proof}[Proof of Theorem~\ref{t:cayley}]
Let $M=M^{15}$ be a connected Einstein hypersurface in $\tM$, where $\tM$ is one of the spaces $\mathbb{O}P^2$ or $\mathbb{O}H^2$. We normalise the metric on $\tM$ in such a way that the maximal absolute value of the curvature is $1$ and denote $\ve=\pm1$ the sign of the curvature of $\tM$.

Let $x \in M$ and let $\xi$ be a unit normal vector to $M$ at $x$. Denote $X_i, \, i=1, \dots, 15$, an orthonormal basis of principal vectors at $x$, with $\la_i$ the corresponding principal curvatures and denote $H = \sum_i \la_i$ the mean curvature.

By Gauss equations, $\tR(X_i,X_k,X_k,X_j) = R(X_i,X_k,X_k,X_j) + (\la_k^2 \delta_{ik}\delta_{jk}-\la_i \la_k\delta_{ij})$. Summing up by $k$ we obtain
  \begin{equation}\label{eq:Gauss}
    \tR(X_i,\xi,\xi,X_j)=(-\la_i^2 + H \la_i + C)\delta_{ij},
  \end{equation}
where $C$ is the difference of the Einstein constants of $\tM$ and of $M$.

It follows from~\eqref{eq:Gauss} that $M$ is a curvature-adapted hypersurface. Recall that a submanifold is called \emph{curvature-adapted} if for every normal, its tangent space is invariant under the corresponding normal Jacobi operator, and the restriction of the latter to the tangent space commutes with the shape operator relative to that normal. Such submanifolds are extensively studied in the literature; we refer the reader to \cite{Ber,Mur} and the references therein. In particular, in \cite{Mur}, the author introduced several classes of curvature-adapted hypersurfaces in $\bO P^2$ and in $\bO H^2$ (and conjectured that there are no others) and proved that if such hypersurface in $\bO P^2$ is complete and has constant principal curvatures, then it is a principal orbit of a cohomogeneity one action.

In our case, the restriction of $\tR_\xi$ to $T_xM$ has two eigenvalues, $\ve$ and $\frac{\ve}{4}$, of multiplicities $7$ and $8$ respectively. We denote $L_\ve$ and $L_{\frac{\ve}{4}}$ the corresponding eigenspaces. Then up to relabelling, \eqref{eq:Gauss} gives
  \begin{equation}\label{eq:ca}
    -\la_i^2 + H \la_i + C = \ve, \quad \text{for } 1 \le i \le 7; \qquad -\la_i^2 + H \la_i + C = \tfrac{\ve}{4}, \quad \text{for } 8 \le i \le 15.
  \end{equation}

  \begin{lemma} \label{l:finite}
    Equations~\eqref{eq:ca} have a finite number of solutions $(\la_1, \dots, \la_{15}) \in \br^{15}$.
  \end{lemma}
  \begin{proof}
    For $i=1, \dots, 7$, let $\la_i=\frac12 H + \frac12 \sqrt{H^2+(4C-4\ve)}$ for $q_1$ values of $i$, and $\la_i=\frac12 H - \frac12 \sqrt{H^2+(4C-4\ve)}$ for $q_2$ values of $i$, where $q_1, q_2 \ge 0$ and $q_1+q_2=7$. Similarly, for $i=8, \dots, 15$, let $\la_i=\frac12 H + \frac12 \sqrt{H^2+(4C-\ve)}$ for $q_3$ values of $i$, and $\la_i=\frac12 H - \frac12 \sqrt{H^2+(4C-\ve)}$ for $q_4$ values of $i$, where $q_3, q_4 \ge 0$ and $q_3+q_4=8$.

    It suffices to show that $H$ can take only finite number of values. From $H=\sum_{i=1}^{15} \la_i$ we get $13 H = (q_2-q_1)\sqrt{H^2+(4C-4\ve)} + (q_4-q_3)\sqrt{H^2+(4C-\ve)}$. Clearing the radicals we obtain the following biquadratic equation for $H$:
    $\big((169-(q_2-q_1)^2-(q_4-q_3)^2)H^2-((q_2-q_1)^2(4C-4\ve)+(q_4-q_3)^2(4C-\ve))\big)^2 = 4(q_2-q_1)^2(q_4-q_3)^2(H^2+(4C-4\ve))(H^2+(4C-\ve))$. If it is satisfied for infinitely many values $H \in \br$, it must be satisfied identically. If $(q_2-q_1)(q_4-q_3) \ne 0$, then the right-hand side must be a square of a polynomial in $H^2$ which is only possible when $4C-4\ve=4C-\ve$; this is a contradiction, as $\ve = \pm1$. If $(q_2-q_1)(q_4-q_3) = 0$, then $q_4=q_3=4$ (as $q_2+q_1=7$) and we obtain $(169-(q_2-q_1)^2)H^2-(q_2-q_1)^2(4C-4\ve) = 0$, and so $q_2-q_1 = \pm 13$, which is again a contradiction.
  \end{proof}

From Lemma~\ref{l:finite} it follows that the principal curvatures of $M$ are constant and have constant multiplicities. We now extend $\xi$ to a smooth unit normal vector field on a neighbourhood $U \subset M$ of $x$ and $X_i, \, i=1, \dots, 15$, to a smooth local orthonormal frame of principal vector fields on $U$, with $\la_i$ the corresponding (constant) principal curvatures. Codazzi equations give
  \begin{equation}\label{eq:codazziO}
    \tR(X_k,X_i,X_j,\xi)  = (\la_i-\la_j) \<\nabla_k X_i, X_j\> - (\la_k-\la_j) \<\nabla_i X_k, X_j\>,
  \end{equation}
where we abbreviate $\nabla_{X_i}$ to $\nabla_i$.

We now differentiate equations~\eqref{eq:Gauss}. Using the fact that $\tM$ is locally symmetric and that $\la_i$ are constant, we get from \eqref{eq:Gauss} and \eqref{eq:codazziO}:
  \begin{equation}\label{eq:difGauss}
    (\la_i-\la_j) (\la_i+\la_j-2\la_k -H) \<\nabla_k X_i, X_j\> + \la_k (\la_j-\la_k) \<\nabla_i X_j, X_k\> + \la_k (\la_k-\la_i) \<\nabla_j X_k, X_i\> = 0,
  \end{equation}
for $i,j,k = 1, \dots, 15$.

Label the principal curvatures as in \eqref{eq:ca}. Then for $i=1, \dots, 7$, the principal curvatures $\la_i$ can take no more than two values, $\alpha_1$ and $\alpha_2$, with the corresponding multiplicities $p_1$ and $p_2$; we have $p_1+p_2 = 7$ and we label them in such a way that $p_1 \ge  p_2 \ge 0$. Similarly, for $8 \le i \le 15, \; \la_i$ can take no more than two values, $\alpha_3$ and $\alpha_4$, with multiplicities $p_3$ and $p_4$ respectively, such that $p_3+p_4 = 8$; label them in such a way that $p_3 \ge  p_4 \ge 0$. Note that $\alpha_s, \; s=1,2,3,4$, are pairwise distinct; denote $E_s$ the corresponding eigenspaces (eigendistributions). We have $L_\ve=E_1 \oplus E_2$ and $L_{\frac{\ve}{4}}=E_3 \oplus E_4$ (note that $E_2$ and $E_4$ can be trivial). 

Take $X_k =X_j \in E_s$ and $\la_i \ne \la_k = \la_j$ in~\eqref{eq:codazziO}. We obtain $\tR(X_i,X_k,X_k,\xi)  = (\la_i-\la_k)$ $\<\nabla_k X_k, X_i\>$. As $\tM$ is rank-one symmetric, its curvature tensor has the ``duality property": a vector $Y$ is an eigenvector of $\tR_X$ if and only if $X$ is an eigenvector of $\tR_Y$. It follows that $\tR(X_i,X_k,X_k,\xi) = 0$ and so $\<\nabla_k X_k, X_i\>=0$. But an orthonormal basis of principal vectors lying in $E_s$ can be chosen arbitrarily, so for any $X \in E_s$, we have $\<\nabla_X X, X_i\>=0$, and hence $\<\nabla_k X_j, X_i\> + \<\nabla_j X_k, X_i\> = 0$, for all $X_i, X_j, X_k$ with $\la_i \ne \la_k = \la_j$. But then from \eqref{eq:difGauss} with $\la_i \ne \la_k = \la_j$ we get $(\la_i-\la_j-H) \<\nabla_k X_j, X_i\> + \la_j \<\nabla_j X_k, X_i\> = 0$ which gives
  \begin{equation}\label{eq:alphalpha}
    (\la_i-2\la_j-H) \<\nabla_k X_j, X_i\> = 0, \quad \text{for } \la_i \ne \la_k = \la_j.
  \end{equation}

\smallskip

Recall that the algebra of octonions $\bO$ is an $8$-dimensional division algebra over $\br$; it is non-associative and non-commutative. For an octonion $a$, its conjugate is given by $a^*=2\<a,1\>1-a$. Note that $aa^*=a^*a=\|a\|^21$. An octonion $a$ is called imaginary, if $a \perp 1$, and unit, if $\|a\|=1$. Denote $\mathcal{L}_a$ (respectively $\mathcal{R}_a$) the operator of left (respectively right) multiplication by $a \in \bO$. Then $\mathcal{L}_a^*=\mathcal{L}_{a^*}, \; \mathcal{R}_a^*=\mathcal{R}_{a^*}$, and the operators $\mathcal{L}_a, \mathcal{R}_a$ are skew-symmetric for imaginary $a$ and orthogonal for unit $a$. The maps $a \to \mathcal{L}_a$ and $a \to \mathcal{R}_a$ define two non-isomorphic representations of $\mathrm{Cl}(7)$ on $\br^8$.

The tangent space $T_x\tM$ can be identified with $\bO \oplus \bO$, so that its elements are the pairs $(a,b)$ of octonions, with the inner product $\<(a,b),(c,d)\>= \<a,c\>+\<b,d\>$. From \cite{BG,EPS}, the curvature tensor is given by
  \begin{multline}\label{eq:curvO}
  \tR ((a,b), (c,d))(e,f) = \tfrac{\ve}{4} (4\<c,e\>a-4\<a,e\>c + (ed)b^* - (eb)d^* + (ad - cb)f^*, \\
  4\<d,f\>b - 4\<b,f\>d + a^*(cf) - c^*(af) - e^*(ad - cb)).
  \end{multline}

As the isotropy subgroup acts transitively on $T_x\tM$ we can take $\xi=(1,0)$. Then from~\eqref{eq:curvO} we obtain $\tR_\xi(a,b)=\tfrac{\ve}{4} (4a-4\<a,1\>1 , b)$, and so the eigenspaces of $\tR_\xi$ are
  \begin{equation*}
    L_\ve = \{(a,0) \, | \, a \in \bO, \, a \perp 1\}, \qquad L_{\frac{\ve}{4}}= \{(0,b) \, | \, b \in \bO\}.
  \end{equation*}
Then \eqref{eq:curvO} gives
  \begin{equation}\label{eq:RLLL}
    \tR(L_\ve, L_\ve) L_\ve, \; \tR(L_{\frac{\ve}{4}}, L_\ve) L_\ve, \; \tR(L_{\frac{\ve}{4}}, L_{\frac{\ve}{4}}) L_{\frac{\ve}{4}} \perp \xi,
  \end{equation}
and for $X=(a,0) \in L_{\ve}$ and $Y=(0,d), Z=(0,f) \in L_{\frac{\ve}{4}}$,
  \begin{equation}\label{eq:adf}
    \tR(X,Y,Z,\xi) = \tfrac{\ve}{4} \<(ad)f^*,1\> = \tfrac{\ve}{4} \<ad,f\> = \tfrac{\ve}{4} \<\mathcal{L}_ad,f\>.
  \end{equation}
It now follows from \eqref{eq:RLLL} and \eqref{eq:codazziO} that
  \begin{equation*}
  \nabla_{E_1}E_1 \subset E_1, \quad \nabla_{E_2}E_2 \subset E_2, \quad \nabla_{E_3}E_3 \perp E_4, \quad \nabla_{E_4}E_4 \perp E_3.
  \end{equation*}

  \begin{lemma}\label{l:78717}
    We have $\alpha_1 = 2 \alpha_3 + H$. Moreover, $E_1=L_\ve, \; E_2=0$ \emph{(}so that $p_1=7, \, p_2 = 0$\emph{)} and one of two cases may occur:
    \begin{enumerate}[label=\emph{(\arabic*)},ref=\arabic*]
      \item \label{it:78}
      $E_3 = L_{\frac{\ve}{4}}, \; E_4=0$ \emph{(}so that $p_3=8, \, p_4 = 0$\emph{)}.

      \item \label{it:717}
      $L_{\frac{\ve}{4}}=E_3 \oplus E_4$, with $p_3=7$ and $p_4 = 1$.
    \end{enumerate}
  \end{lemma}
  \begin{proof}
    Note that $p_3 \ge p_4$ and $p_3+p_4=8$, so $p_3 \ge 4$.

    First suppose that $p_3 > 4$. If there exists $s =1, 2$ such that $\alpha_s \ne 2 \alpha_3 + H$, then taking $X_i \in E_s$ and $X_j, X_k \in E_3$ in~\eqref{eq:alphalpha} we get $\<\nabla_k X_j, X_i\> = 0$, and so $\tR(X_i,X_k,X_j,\xi) = 0$ by~\eqref{eq:codazziO}. But then $X_i \in L_\ve, X_j, X_k \in L_{\frac{\ve}{4}}$ and so by~\eqref{eq:adf} with $X=X_i=(a,0), \, a \perp 1$, and $Y=X_k=(0,d), \, Z=X_j=(0,f)$, we get $\<\mathcal{L}_ad,f\>=0$. But as the octonion $a$ is unit and imaginary, the operator $\mathcal{L}_a$ on $\bO$ is both orthogonal and skew-symmetric. Then its maximal isotropic subspace has dimension $4$ which contradicts the fact that $p_3 > 4$. From this contradiction we obtain $p_2=0, \, p_1=7$ and $\alpha_1 = 2 \alpha_3 + H$. Furthermore, if $p_4 \ne 0$, then $\alpha_1 \ne 2 \alpha_4 + H$ and a similar argument shows that $\<ad,f\>=0$, for all $a \perp 1$ and for all $d,f$ in a subspace of $L_{\frac{\ve}{4}}$ of dimension $p_4$. But then we have $\<a,fd^*\>=0$, and so $fd^*$ is real which is only possible when $d$ and $f$ are (real) proportional. It follows that $p_4=1$.

    Now consider the case $p_3=4$. Then $p_4=4$ and we have $\alpha_1 \ne 2 \alpha_s + H$ for at least one of $s=3,4$, say for $s=3$. Denote $V_3$ and $V_4$ the projections of $E_3$ and $E_4$ to the second copy of $\bO$ in the decomposition $T_x\tM=\bO \oplus \bO$. Then $V_3$ and $V_4$ are orthogonal subspaces of $\bO$ of dimension $4$. Similarly, denote $V_1$ and $V_2$ the projections of $E_1$ and $E_2$ to the first copy of $\bO$ in $T_x\tM=\bO \oplus \bO$. The subspaces $V_3$ and $V_4$ are orthogonal, of dimensions $p_3$ and $p_4$ respectively, with $V_1 \oplus V_2 = \bO \cap 1^\perp$. Repeating the arguments from the previous paragraph we find that $\<\mathcal{L}_ad,f\>=0$, for all $d, f \in V_3$ and all $a \in V_1$. For a unit octonion $a \in V_1$, the operator $\mathcal{L}_a$ on $\bO$ is both orthogonal and skew-symmetric. It has an isotropic subspace $V_3$ of dimension $4$, and hence the complementary subspace $V_4$ must also be isotropic, so $\<\mathcal{L}_ad,f\>=0$, for all $d, f \in V_4$ and all $a \in V_1$. But if $p_2 \ne 0$, then $\alpha_2 \ne 2 \alpha_s + H$ for at least one of $s=3,4$. Repeating the argument above for $\alpha_2$ we obtain $\<\mathcal{L}_ad,f\>=0$, for all $a \in V_2$ and all $d, f \in V_s$, where $s=3,4$. It follows that for all $d, f \in V_3$ and for all $a \perp 1$ we have $0=\<\mathcal{L}_ad,f\>=\<ad,f\>= \<a,fd^*\>$. But then $fd^*$ must always be real which contradicts the fact that $p_3 = 4$.
  \end{proof}

In case~\eqref{it:78} of Lemma~\ref{l:78717} we have $\alpha_1 = 2 \alpha_3 + H$ and $H = 7\alpha_1 + 8\alpha_3$. Eliminating $C$ from~\eqref{eq:ca} we obtain $(\alpha_3-\alpha_1)(\alpha_3+\alpha_1 - H) = \tfrac{3}{4}\ve$. Solving these equations we find
  \begin{equation}\label{eq:alphas78}
    \ve=1, \quad \alpha_1 = -\ve' \frac{5\sqrt{6}}{24}, \quad \alpha_3 = \ve' \frac{\sqrt{6}}{8}, \quad \text{where } \ve'=\pm1
  \end{equation}
(the sign of $\ve'$ depends on the direction of $\xi$). In case~\eqref{it:717} we have $\alpha_1 = 2 \alpha_3 + H$ and $H = 7\alpha_1 + 7\alpha_3+\alpha_4$. Combining with~\eqref{eq:ca} and solving the resulting system of equations we obtain
  \begin{equation}\label{eq:alphas771}
    \ve=1, \quad \alpha_1 = -\ve' \frac{6}{2\sqrt{91}}, \quad \alpha_3 = \ve' \frac{7}{2\sqrt{91}}, \quad \alpha_4 = -\ve' \frac{27}{2\sqrt{91}}, \quad \text{where } \ve'=\pm1.
  \end{equation}

Note that in both cases, $\ve=1$. This proves that there are no Einstein hypersurfaces in $\bO H^2$.

Next we show that case~\eqref{it:717} of Lemma~\ref{l:78717} is not possible. Choose arbitrary $X \in E_1, \, Y \in E_3$ and $Z \in E_4$, and substitute in~\eqref{eq:difGauss} first $(\la_i, \la_j, \la_k)=(\alpha_1, \alpha_3, \alpha_4)$ and $(X_i,X_j,X_k)=(X,Y,Z)$, then $(\la_i, \la_j, \la_k) = (\alpha_4, \alpha_1, \alpha_3)$ and $(X_i,X_j,X_k)=(Z,X,Y)$, and then $(\la_i, \la_j, \la_k)=(\alpha_3, \alpha_4, \alpha_1)$ and $(X_i,X_j,X_k)= (Y,Z,X)$. We get a system of linear equations for $\<\nabla_ZX,Y\>, \<\nabla_XY,Z\>$ and $\<\nabla_YZ,X\>$ whose matrix, after substituting the values of $\alpha_s$ from \eqref{eq:alphas771}, is given by
  \begin{equation*}
    Q= \frac{1}{(4\cdot 91)^3}
    \left(
        \begin{array}{ccc}
          -75 \times 13 & -27 \times 34 & 27 \times 21 \\
          6 \times 13 & 12 \times 34 & 6 \times 21 \\
          -7 \times 13 & 7 \times 34 & 27 \times 21 \\
        \end{array}
      \right).
  \end{equation*}
We have $\det Q = -\frac{39051}{16562} \ne 0$, and so $\<\nabla_ZX,Y\>=\<\nabla_XY,Z\>=\<\nabla_YZ,X\>=0$. Then by~\eqref{eq:codazziO} we get $\tR(X,Y,Z,\xi) = 0$. Then for $X=(a,0), \, Y=(0,d), \, Z=(0,f)$, equation~\eqref{eq:adf} gives $\<ad,f\> = 0$, for all imaginary $a$ and all $d \perp f$ (note that as $p_4=1$, the octonion $f$ is fixed up to a real multiple). But as $a$ is imaginary, for $d=af$ we have $d \perp f$ and $ad=a(af)=a^2 f = - \|a\|^2 f$, which is a contradiction with $\<ad,f\> = 0$.

To complete the proof we show that an Einstein hypersurface in $\bO P^2$ whose principal curvatures and principal distributions are as those given in case~\eqref{it:78} of Lemma~\ref{l:78717} is a domain of the geodesic sphere of radius $r_0$ (see Example~\ref{ex:sphereO}). Consider the normal exponential map $\Phi_r: M \to \bO P^2$ defined by $\Phi_r(x)= \exp_x(r\xi)$, where $\xi$ is chosen in such a way that $\ve'= 1$ in the equations~\eqref{eq:alphas78}. The rank of $(\Phi_r)_*$ is the dimension of the span of the values at $r$ of the Jacobi vector fields $F_i$ along the geodesic $r \to \exp_x(r\xi)$ which are defined by the initial conditions $F_i(0)=X_i, \; \on_\xi F_i(0) = - \la_i X_i$ (cf. \cite[5.2]{Ber}). Solving the Jacobi equations we find $F_i(r)=(\cos r + \frac{5\sqrt{6}}{24} \sin r) \overline{X}_i$, for $i=1, \dots, 7$, and $F_i(r)=(\cos \frac12r - \frac{\sqrt{6}}{4} \sin \frac12 r) \overline{X}_i$, for $i=8, \dots, 15$, where $\overline{X}_i$ is the parallel translation of $X_i$ along the geodesic $r \to \exp_x(r\xi)$. It follows that at $r=r_0$, that is, when $\cot r = -\frac{5\sqrt{6}}{24}$, all the Jacobi fields vanish, so that the rank of $(\Phi_{r_0})_*$ is zero. As $M$ is connected, the image of the map $\Phi_{r_0}$ is a single point, and so $M$ lies on the geodesic sphere of radius $r_0$ centred at that point.
\end{proof}

\section{$2$-stein hypersurfaces of a $2$-stein space}
\label{s:2stein}

Suppose $M$ is a $2$-stein hypersurface of a $2$-stein space $\tM$. Then for all $x \in M, \; y \in \tM$ and all $X \in T_x M, \; Y \in T_y\tM$ we have
\begin{equation} \label{eq:22stein1}
  \Tr R_X = c_1 \|X\|^2, \quad \Tr (R_X^2) = c_2 \|X\|^4, \qquad \Tr \tR_Y = \tc_1 \|Y\|^2, \quad \Tr (\tR_Y^2) = \tc_2 \|Y\|^4,
\end{equation}
for some constants $c_1, c_2, \tc_1, \tc_2 \in \br$, where $R$ and $\tR$ are the curvature tensors of $M$ and $\tM$ respectively.

Let $x \in M \subset \tM$ and let $X_1, X_2, \dots, X_{n-1}, X_n$ be an orthonormal basis for $T_x\tM$ such that $X_n$ is orthogonal to $T_xM$. Denote $\Sh$ the shape operator of $M$ at $x$ and $B$ the restriction of $\tR_{X_n}$ to $T_xM$, so that $B$ is the symmetric operator on $T_xM$ defined by $\<BX,X\>=\tR(X,X_n,X_n,X)$. By Gauss equation, for $X \in T_xM, \; t \in \br$ and $i,j<n$ we have
\begin{equation}\label{eq:RXtn}
\begin{split}
  \<\tR_{X+tX_n}X_n,X_n\> & = \<BX,X\>, \\
  \<\tR_{X+tX_n}X_i,X_n\> & = \tR(X_i, X, X, X_n) - t \<BX,X_i\>, \\
  \<\tR_{X+tX_n}X_i,X_j\> & = \<R_{X}X_i,X_j\> - (\<\Sh X,X\>\<\Sh X_i,X_j\>-\<\Sh X,X_i\>\<\Sh {L}X,X_j\>) \\
  & \quad + t(\tR(X_i, X, X_n, X_j) + \tR(X_i, X_n, X, X_j)) +t^2  \<BX_i,X_j\>,
\end{split}
\end{equation}
Then from \eqref{eq:22stein1} with $Y=X+tX_n$ we obtain $\Tr \tR_{X+tX_n} = \tc_1 (\|X\|^2 + t^2)$ and so from \eqref{eq:RXtn},
\begin{gather}\label{eq:EinE1}
  \Tr B = \tc_1, \\
  \sum_{i<n} \tR(X_i, X, X_n, X_i) =0, \label{eq:EinE2} \\
  B=-\Sh^2 + (\Tr \Sh) \Sh +(\tc_1-c_1) \id. \label{eq:EinE3}
\end{gather}
Note that taking the trace in~\eqref{eq:EinE3} we get $(n-1) c_1 - (n-2) \tc_1 = (\Tr \Sh)^2 - \Tr (\Sh^2)$ by~\eqref{eq:EinE1}.

Furthermore, from \eqref{eq:22stein1} with $Y=X+tX_n$ we obtain $\Tr (\tR_{X+tX_n}^2) = \tc_2 (\|X\|^4 + 2 t^2 \|X\|^2 + t^4)$. Using \eqref{eq:RXtn} and collecting the coefficients of the powers of $t$ we get
\begin{gather}
  \Tr (B^2)  = \tc_2, \label{eq:2in21} \\
  \sum_{i<n} \tR(X_i, X, X_n, BX_i)  = 0,  \notag \\ 
  \begin{split}
  \|BX\|^2 + \Tr (R_X B) &- \tfrac12 \<((\Tr (\Sh B) \Sh  - \Sh B\Sh )X,X\> \\
  &+ \sum_{i,j <n} (\tR(X_i, X, X_n, X_j) + \tR(X_j, X, X_n, X_i))^2 = \tc_2 \|X\|^2,  \end{split} \label{eq:2in23} \\
  \sum_{i<n} \tR(X_i, X_n, X, R_X X_i) - \<\Sh X,X\> \sum_{i<n} \tR(X_i, X_n, X, \Sh  X_i) + \<\tR_{\Sh X}X,X_n\> -\<R_XX_n, BX\> = 0,  \notag \\ 
  \begin{split}
  2 \|\tR_XX_n\|^2 & - \<BX,X\>^2 - 2 \<\Sh X,X\> \Tr (R_X \Sh ) + 2 R(\Sh X,X,X,\Sh X)  \\ & + \Tr (\Sh^2) \<\Sh X,X\>^2 + \|\Sh X\|^4 - 2 \<\Sh X,X\> \<\Sh^3X,X\> = (\tc_2-c_2) \|X\|^4.  \notag
  \end{split} 
\end{gather}

\begin{proof}[Proof of Theorem~\ref{t:2st}]
\eqref{it:2sttg} In this case, $\Sh=0$ and so by~\eqref{eq:EinE3}, $B = \rho \id$ for some $\rho \in \br$. Then from \eqref{eq:EinE1} we have $\tc_1 = (n-1) \rho$ and from \eqref{eq:2in21}, $\tc_2 = (n-1) \rho^2$. Then by~\eqref{eq:22stein1}, for any $y \in \tM$ and any $Y \in T_y\tM$ we obtain  $\Tr \tR_Y = (n-1) \rho \|Y\|^2$ and $\Tr (\tR_Y)^2 = (n-1) \rho^2 \|Y\|^4$, and so by Cauchy-Schwartz inequality, the restriction of $\tR_Y$ to the subspace $Y^\perp$ equals $\rho$ times the identity operator. So $\tM$ is of constant curvature, and then $M$ is also of (the same) constant curvature.

\eqref{it:2stcc} We have $B = \rho \id$. Then from (\ref{eq:EinE1}, \ref{eq:2in21}) we obtain $\tc_1 = (n-1) \rho, \; \tc_2 = (n-1) \rho^2$. Furthermore, in equation \eqref{eq:2in23}, we have $\tR(X_i, X, X_n, X_j)=0$ and $\Tr (\Sh B) \Sh - \Sh B \Sh= \rho(c_1-(n-2)\rho)$ by~\eqref{eq:EinE3} and so we get $c_1=(n-2)\rho$. Then from~\eqref{eq:EinE3} it follows that $\Sh^2=(\Tr \Sh) \Sh$, and so $\operatorname{rk} \Sh \le 1$. But then taking $t=0$ in the third equation of~\eqref{eq:RXtn} we find that $M$ has the same constant curvature as $\tM$.
\end{proof}




\begin{thebibliography}{BPV} 

\bibitem[ABS]{ABS}
M.\,F.\,Atiah, R.\,Bott, A.\,Shapiro, \emph{Clifford modules,} Topology, \textbf{3, suppl.1} (1964), 3--38.

\bibitem[Ber]{Ber}
J.\,Berndt, \emph{Real hypersurfaces in quaternionic space forms}, J. Reine Angew. Math. \textbf{419} (1991), 9--26.

\bibitem[BTV]{BTV}
J.\,Berndt, F.\,Tricerri, L.\,Vanhecke, \emph{Generalized Heisenberg groups and Damek-Ricci harmonic spaces}. Lecture Notes in Mathematics, 1598. Springer-Verlag, Berlin, 1995.

\bibitem[BPV]{BPV}
J.\,Berndt, F.\,Prufer, L.\,Vanhecke, \emph{Totally geodesic submanifolds of symmetric-like Riemannian manifolds}, Tsukuba J. Math.
\textbf{22} (1998), 463--475.

\bibitem[BG]{BG}
R.\,Brown, A.\,Gray, \emph{Riemannian manifolds with holonomy group $\mathrm{Spin}(9)$}. In: Differential Geometry in
honor of K.Yano, pp. 41--59. Kinokuniya, Tokyo (1972).

\bibitem[Car]{Car}
E.\,Cartan, \emph{G\`{e}ometrie des \`{e}spaces de Riemann}, Paris, 1946.

\bibitem[CR]{CR}
T.\,Cecil, P.\,Ryan, \emph{Geometry of Hypersurfaces}. Springer Monographs in Mathematics. Springer-Verlag, New York, 2015.

\bibitem[DR]{DR}
E.\,Damek, F.\,Ricci, \emph{A class of nonsymmetric harmonic Riemannian spaces}, Bull. Amer. Math. Soc. (N.S.) \textbf{27} (1992), 139--142.

\bibitem[EPS]{EPS}
Y.\,Euh, J.\,H.\,Park, K.\,Sekigawa, \emph{Characteristic function of Cayley projective plane as a harmonic manifold}, Hokkaido Math. J. \textbf{47} (2018), 191--203.

\bibitem[Fia]{Fia}
A.\,Fialkow, \emph{Hypersurfaces of a Space of Constant Curvature}, Ann. of Math. \textbf{39} (1938), 762--785.

\bibitem[Heb]{Heb}
J.\,Heber, \emph{On harmonic and asymptotically harmonic homogeneous spaces}, Geom. Funct. Anal. \textbf{16} (2006), 869--890.

\bibitem[Kn]{Kn}
G.\,Knieper,
\emph{A survey on noncompact harmonic and asymptotically harmonic manifolds}, Geometry, topology, and dynamics in negative curvature, 146--197,
London Math. Soc. Lecture Note Ser., 425, Cambridge Univ. Press, Cambridge, 2016.

\bibitem[Kon]{Kon}
M.\,Kon, \emph{Pseudo-Einstein real hypersurfaces in complex space forms}, J. Differ. Geom. \textbf{14} (1979), 339--354.

\bibitem[MP]{MP}
A.\,Martinez, J.D.\,P\'{e}rez, \emph{Real hypersurfaces in quaternionic projective space}, Ann. Mat. Pura Appl. (4) \textbf{145} (1986), 355--384.

\bibitem[Mon]{Mon}
S.\,Montiel, \emph{Real hypersurfaces of a complex hyperbolic space}, J. Math. Soc. Japan \textbf{37} (1985), 515--535.

\bibitem[Mur]{Mur}
T.\,Murphy, \emph{Curvature-Adapted Submanifolds of Symmetric Spaces}, Indiana Univ. Math. J. \textbf{61} (2012), 831--847.

\bibitem[N1]{N1}
Y.\,Nikolayevsky, \emph{Osserman manifolds of dimension $8$}, manuscripta math. \textbf{115} (2004), 31--53.

\bibitem[N2]{N2}
Y.\,Nikolayevsky, \emph{Osserman conjecture in dimension $n\neq 8,16$}, Math. Ann. \textbf{331} (2005), 505--522.

\bibitem[N3]{N3}
Y.\,Nikolayevsky, \emph{Two theorems on harmonic manifolds}, Comm. Math. Helv. \textbf{80} (2005), 29--50.

\bibitem[NP]{NP}
Y.\,Nikolayevsky, J.\,H.\,Park, \emph{$H$-contact unit tangent sphere bundles of Riemannian manifolds}, Diff. Geom. Appl. \textbf{49} (2016), 301--311.

\bibitem[OP]{OP}
M.\,Ortega, J.D.\,P\'{e}rez, \emph{On the Ricci tensor of a real hypersurface of quaternionic hyperbolic space}, Manuscripta Math. \textbf{93} (1997), 49--57.

\bibitem[Rou]{Rou}
F.\,Rouvi\`{e}re, \emph{X-ray transform on Damek-Ricci spaces}, Inverse Probl. Imag. \textbf{4} (2010), 713--720.

\bibitem[Sz]{Sz}
Z.\,I.\,Szab\'{o}, \emph{The Lichnerowicz conjecture on harmonic manifolds}, J. Differential Geom. \textbf{31} (1990), 1--28.

\bibitem[VW]{VW}
L.\,Vanhecke, T.\,J.\,Willmore \emph{Interaction of tubes and spheres}, Math. Ann. \textbf{263} (1983), 31--42.

\end{thebibliography}
\end{document}